\documentclass[a4paper,10pt,reqno, english]{amsart}

\usepackage[english]{babel}
\usepackage[T1]{fontenc}

\usepackage[a4paper,top=3cm,bottom=2cm,left=3cm,right=3cm,marginparwidth=1.75cm]{geometry}
\usepackage{graphicx}
\graphicspath{{.}}

\usepackage{tikz}
\usetikzlibrary{calc, through, intersections}

\usepackage[utf8]{inputenc}
\usepackage{listings}
\usepackage{xcolor}
\usepackage{fancyhdr}
\lstdefinestyle{Python}{
    language         = Python,
    numbers          = left,
    showspaces       = false,
    showstringspaces = false,
    basicstyle       = \ttfamily,
    keywordstyle     = \color{blue},
    stringstyle      = \color{orange},
    commentstyle     = \color{red}\ttfamily
}

\usepackage{enumitem}

\usepackage{amsmath}
\usepackage{graphicx}
\usepackage{amsthm}
\usepackage{amssymb}
\usepackage[colorinlistoftodos]{todonotes}
\usepackage[colorlinks=true, allcolors=blue]{hyperref}

\newtheorem{theorem}{Theorem}[section]
\newtheorem{lemma}[theorem]{Lemma}
\newtheorem{proposition}[theorem]{Proposition}

\newtheorem{corollary}[theorem]{Corollary}
\newtheorem{definition}[theorem]{Definition}

\newcommand{\Fq}{\mathbb{F}_q}
\newcommand{\PFq}{\mathbb{P}^2(\Fq)}
\newcommand{\D}{\mathcal{D}}

\title{Counting 7-Arcs in Projective Planes over Finite Fields}
\author{Andrei Staicu}
\date{}

\begin{document}

\maketitle

\begin{abstract}
    
Given a collection of points in the plane, classifying which subsets are collinear is a natural problem and is related to classical geometric constructions. We consider collections of points in a projective plane over a finite field such that no three are collinear. This is a finite set and its size is both combinatorially interesting and has deeper topological consequences. We count the number of such collections classified by the algebraic symmetries of the finite field. Variations of this problem have been considered by Glynn, Bergvall, Das, O'Connor et al. We obtain the counts for 7 points over fields of characteristic 2. These new counts are governed by the existence and classification of a configuration of points called the Fano plane.
    
\end{abstract}

\section{Introduction}

Denote the finite field of $q$ elements by $\Fq$. When considering the projective plane over $\Fq$, we can count the size of the set of $n$-tuples of points in $\PFq$ such that no three points are collinear. This count has been computed for $n \le 8$ by Glynn \cite{glynn1988rings}, and $n = 9$ by Iampolskaia, Skorobogatov, and Sorokin \cite{iampolskaia1995formula}. It is known that the size of this set is given by a polynomial in $q$ for $n \le 6$ and is \textit{quasipolynomial} for $7 \le n \le 9$, see Definition \ref{QDef}.

In this paper, we will consider an extension of this problem. Let $F_q$ be the automorphism of $\Fq$ defined by $x \mapsto x^q$, we will denote this by $F$ when $q$ is understood.

\begin{definition}
    An n-arc defined over $\Fq$ is a set $S$ of size $n$ of points in $\mathbb{P}^2(\overline{\Fq})$ such that no three points of $S$ are collinear and $S$ is $F_q$-invariant, i.e. for all $a \in S$ we have that $F_q(a) \in S$.
\end{definition}

Let $B_n$ be the collection of unordered $n$-arcs over $\Fq$. Notice that $F_q$ acts on each element of $B_n$ by permuting its coordinates. For each conjugacy class of $S_n$, which we index by its cycle type, the partition $\lambda$ of $n$, let $B_n^\lambda \subset B_n$ be the set of $n$-arcs $S$ in $B_n$ such that the action of $F_q$ on $S$ gives the cycle type $\lambda$. Note that for any labeling of $S \in B_n$, the action of $F_q$ on $S$ gives the same cycle type, thus $B_n^\lambda$ is a well defined subset of $B_n$.

These counts have been computed for $n \le 6$ by Das and O'Connor \cite{das2021configurations}, and for $n = 7$ in fields of characteristic $p > 2$ by Bergvall \cite{bergvall2020cohomology}. These counts were used to deduce topological consequences. The counts computed in this paper can be used the same way, but we leave this to the expert reader. The main result of this paper (see section \ref{S4}) is to extend the counts for $n = 7$ to include characteristic $2$.

\begin{theorem} \label{1.21}
    The counts for $7$-arcs that are different for characteristic $p = 2$ are listed below:
\begin{table}[h]
\centering
\begin{tabular}{ l | l }
Cycle Type $\lambda$ & $|B_7^\lambda|/|\text{PGL}(3, \Fq)|$\\
\hline
 $e$ & $\frac{1}{5040}(q^6 - 28q^5 + 323q^4 - 1952q^3 + 6462q^2 - 11004q + 7440)$\\ 
 $(12)(34)$ & $\frac{1}{48}(q^6 - 4q^5 - q^4 + 16q^3 - 6q^2 - 12q)$\\ 
 $(123)(456)$ & $\frac{1}{18}(q^6 - q^5 - q^4 - 8q^3 + 9q^2 + 6q + 12)$\\
 $(1234)(56)$ & $\frac{1}{8}(q^6 - 3q^4)$\\ 
 $(1234567)$ & $\frac{1}{7}(q^6 + q^4 + q^3 + q^2 -1)$
\end{tabular}
\caption{$7$-Arc counts for characteristic $2$ fields}
\end{table}
\end{theorem}

\section{Preliminaries}

Let $\mathbb{P}^{2\vee}(\Fq)$ denote the collection of lines in $\PFq$. Note that for $(a : b : c) = p \in \PFq$, we have $F_q(a : b : c) = (F_qa : F_qb : F_qc)$ and for $\{ax + by + cz\} = \ell \in \mathbb{P}^{2\vee}(\Fq)$, $F_q\{ax + by + cz = 0\} = \{ (F_qa)x + (F_qb)y + (F_qc)z = 0\}$. For distinct $p, p' \in \PFq$, let $\langle p, p' \rangle$ be the line $\ell$ such that $\{p, p' \} \subset \ell$. Similarly for distinct $\ell, \ell' \in \mathbb{P}^{2\vee}$ define $\langle \ell, \ell' \rangle$ as the point $p \in \PFq$ such that $\{ p \} = \ell \cap \ell'$. One can check that $\langle p, p' \rangle$ and $\langle \ell, \ell' \rangle$ are indeed unique.

\begin{definition}
    A $q^n$-point is a point $p \in \mathbb{P}^2(\overline{\Fq})$ such that $p \in \mathbb{P}^2(\mathbb{F}_{q^n})$ for $n$ minimal. A $q^n$-line is a line $\ell \in \mathbb{P}^{2\vee}(\overline{\Fq})$ such that $p \in \mathbb{P}^{2\vee}(\mathbb{F}_{q^n})$ for $n$ minimal
\end{definition}

We denote $D_q^n$ to be the set of all $q^n$-points. Notice that $D^1_q = \PFq$, from Galois theory of $\Fq$ we have the following recursive definition for $D^n_q$:
\begin{align}
    D_q^n = \mathbb{P}^2(\mathbb{F}_{q^n}) \setminus \bigcup_{\substack{d \mid n; d \neq n}} D_q^d
\end{align}
Since we know that $|D_q^1| = |\PFq| = q^2 + q + 1$, we can compute the size of $D_q^n$ using the following formula:
\begin{align}
    |D_q^n| = \left|\mathbb{F}_{q^n}\right| - \left|\bigcup_{d \mid n; d \neq n} D_q^d\right|
\end{align}

For each cycle type $\lambda$ of $S_n$, we will define $C_n^\lambda$ as the collection of unordered $n$-tuples of points in $\mathbb{P}^2(\overline{\mathbb{F}_q})$ that are $F_q$-invariant where the action of $F_q$ on every tuple gives the cycle type $\lambda$. It is useful to characterize $C_n^\lambda$ since $B_n^\lambda \subset C_n^\lambda$.

Take some $S \in C_n^\lambda$, since $F_q(S) = S$, then the orbit $O(a)$ of a point $a \in S$ is contained in $S$. Furthermore, if $\lambda$ represents the partition $\{r_1, \ldots, r_\ell \}$ of $n$ then we know that $|O(a)| = r_i$ for some $i$. Let $\sim$ be the equivalence relation where for $a, a' \in D_q^n$, $a \sim a'$ if $a' \in O(a)$. It can be checked this is indeed an equivalence relation, now let $\mathcal{D}_q^n = D_q^n / \sim$.

We can write the partition $\lambda$ as $\{(m_1, s_1), \ldots, (m_1, s_k)\}$ where $m_i$ is the multiplicity of $s_i \in \lambda$ and all $s_i$ are distinct. For every tuple $p \in \text{Conf}_{m_1}(\D_q^{s_1}) \times \cdots \times \text{Conf}_{m_k}(\D_q^{s_k})$, the set defined by $S = \cup_{i}\cup_j^{m_i} O(p_{i,j})$ is a $n$-tuple that is $F_q$-invariant by construction and gives the cycle type $\lambda$. Notice that this map is onto, but not injective since each configuration space is a set of ordered points and $C_7^\lambda$ is not. However, by a simple counting argument, we can compute the size of each pre-image of $S \in C_7^\lambda$ under this map to conclude:
\begin{align}\label{Conf1}
    \left| \text{Conf}_{m_1}(\D_q^{s_1}) \times \cdots \times \text{Conf}_{m_k}(\D_q^{s_k}) \right| = \left(\prod_{i = 1}^k m_i!\right)|C_n^\lambda|
\end{align}

We will use this fact in our study of $B_n^\lambda$ and $C_n^\lambda$.

The work done by Glynn \cite{glynn1988rings} and Iampolskaia, Skorobogatov, and Sorokin \cite{iampolskaia1995formula} can be viewed as the counts of the trivial cycle type for $n \le 9$, i.e. the count corresponding to the conjugacy class of $e$. Glynn \cite[Theorems 4.2 and 4.4]{glynn1988rings} computes have the following count for $|B_7^e|$:
\begin{multline}
    |B_7^e| = (q^2 + q + 1)(q^2 + q)q^2(q - 1)^2((q - 3)(q - 5)(q^4 - 20q^3 +\\ 148q^2 - 468q + 498) - 30a(q))
\end{multline}
And the following explicit formula for $a(q)$, let $q = p^s$:
\begin{align}
a(q) &= 
\begin{cases}
        0 & \text{if } p = 2\\
        1 & \text{otherwise}\\
\end{cases}
\end{align}

From Glynn, we know that $a(q)$ gives the number of Fano planes. We call a Fano plane a superconfiguration since it are configurations of $7$ points and $7$ lines. It is clear that the count above is not polynomial, but are only off by some periodicity. This motivates our use of the term quasipolynomial, defined below.

\begin{definition}\cite[Definition 1.14]{kaplan2017counting}\label{QDef}
    A quasipolynomial of period $m$ is a function $g(x)$ of positive integers such that there is a collection of polynomials $f_0(x), \ldots, f_{m-1}(x)$ such that $g(x) = f_i(x)$ for all $x \equiv i \mod m$.
\end{definition}

\section{Fano Plane}

First we will precisely define the Fano plane:

\begin{definition}
    A Fano plane is a collection of $7$ points and $7$ lines such that $3$ points are on every line and $3$ lines pass through every point
\end{definition}
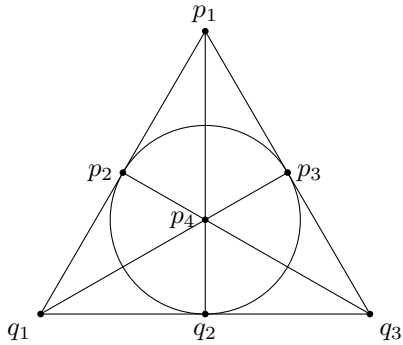
\begin{figure}[h!t]
\centering
\begin{tikzpicture}[scale=2.5]
  \coordinate[label=above:$p_1$] (p1) at (90:1);
  \coordinate[label=below left:$q_1$] (q1) at (210:1);
  \coordinate[label=below right:$q_3$] (q3) at (-30:1);

  \draw (p1) -- (q1) -- (q3) -- cycle;

  \coordinate[label=left:$p_2$] (p2) at ($(p1)!0.5!(q1)$);
  \coordinate[label=below:$q_2$] (q2) at ($(q1)!0.5!(q3)$);
  \coordinate[label=right:$p_3$] (p3) at ($(q3)!0.5!(p1)$);
  \coordinate[label=left:$p_4$] (p4) at (barycentric cs:p1=1,q1=1,q3=1);

  \draw (p2) -- (q3);
  \draw (p3) -- (q1);
  \draw (p1) -- (q2);

  \foreach \point in {p1, p2, p3, q1, q2, q3, p4}
    \fill (\point) circle (0.5pt);

  \node [draw, circle through=(q2), black] at (p4) {};
\end{tikzpicture}
\caption{A diagram of a Fano plane}
\label{fig:FP}
\end{figure}

In going through the proofs for all the counts of $7$-arcs in characteristic $p > 2$ in Bergvall \cite{bergvall2020cohomology}, there is only one claim that does not hold true in characteristic $p = 2$. It states: for any points $p_1, p_2, p_3, p_4 \in \PFq$, the intersections $q_1 = \langle \langle p_1, p_2 \rangle, \langle p_3, p_4 \rangle \rangle$, $q_2 = \langle \langle p_1, p_3 \rangle, \langle p_2, p_4 \rangle \rangle$, and $q_3 = \langle \langle p_1, p_4 \rangle, \langle p_2, p_3 \rangle \rangle$ are not collinear. In fact, we know that every $4$-arc is contained in a unique subplane of order $p$ if $q = p^d$ \cite[Theorem 4.5]{glynn1988rings}. Therefore, in characteristic $2$, $p_1, p_2, p_3, p_4$ are contained in the subplane of order $2$, the Fano plane which contains $7$ points, so $q_1, q_2, q_3$ are always collinear. We will call this the Fano axiom.

\begin{lemma}
    If there is no $S \in C_7^\lambda$ which forms a Fano plane then the count for characteristic $2$ for $|B_7^\lambda|$ is the same as for characteristic $p > 2$.
\end{lemma}
\begin{proof}
    It follows that if there is no $S \in C_7^\lambda$ which forms a Fano plane, then the same argument for counting the size of $|B_7^\lambda|$ works as given in Bergvall since the assumption that is not true in characteristic $2$ is not used, and thus the count for characteristic $2$ for $|B_7^\lambda|$ stays same. 
\end{proof}

This gives an incentive to classify all cycle types $\lambda$ of $S_7$ where there is some $S \in C_7^\lambda$ that is a Fano plane; we say that $\lambda$ admits a Fano plane construction. First, we classify all the cycle types which do not admit a Fano plane construction.

\begin{theorem}\label{3.2T}
    All cycle types of $S_7$ other than $(1234567)$, $(1234)(56)$, $(123)(456)$, $(12)(34)$, and $e$, do not admit a Fano plane constructions.
\end{theorem}

Initially it seems that the study of the Fano plane would require a lot of casework. However, there are only $|\text{PGL}(3, \mathbb{F}_2)| = 7$ different $4$-arcs in any instance of the Fano plane, and it can be checked that given some Fano plane $P \subset \mathbb{P}^2(\mathbb{F}_q)$, for all $4$-arc subsets $\{p_1, \ldots, p_4\} = A \subset P$ we have that $A \cup \{ \langle \langle p_1, p_2 \rangle, \langle p_3, p_4 \rangle \rangle, \langle \langle p_1, p_3 \rangle, \langle p_2, p_4 \rangle \rangle, \langle \langle p_1, p_4 \rangle, \langle p_2, p_3 \rangle \rangle\} = P$. Since this set does not depend on how $A$ is labeled, it defines a map from $4$-arcs to Fano planes. We call this the set the Fano plane generated by $A$ and denote it by $P_A$. Using this observation, we can characterize the non-existence of Fano plane constructions as follows.
\begin{proposition}\label{P34}
    The cycle type $\lambda$ does not admit a Fano plane if for all $S \in C_7^\lambda$, either $S$ does not have a $4$-arc subset, or for some $4$-arc subset $A \subset S$ then $P_A \neq S$.
\end{proposition}

\begin{lemma}
For cycle types $\lambda = (123456), (12345)(67), (12345)$ there does not exist $S \in C_7^\lambda$ such that $S$ gives a Fano plane.
\end{lemma}

\begin{proof}
For $\lambda = (123456)$, $O(a) \cup \{b\} = S \in C_7^\lambda$ for some $a \in D^6$ and $b \in D^1$. Either, $O(a)$ forms a $6$-arc, $O(a)$ lies on two conjugate $q^2$ lines which do not intersect in $O(a)$, or $O(a)$ lies entirely on a $q$-line. In the first two cases, let $A$ be any $4$-arc subset $A \subset O(a) \subset S$. In the first case, clearly $P_A \neq S$ since $S$ contains a $4$-arc. In the second case, $P_A \neq S$ since $O(a) \subset S$ lies on two lines but they do not intersect in $O(a)$. In the last case there does not exist a $4$-arc subset of $S$.

For cycle types including a $q^5$ point $a$, we know $O(a)$ is either on a $q$-line or forms a $5$-arc. In the first case, if there is a $4$-arc subset $A \subset S$, $P_A$ doesn't contain a line with $5$ points on it, so $P_A \neq S$. In the second case, for any $4$-arc subset $A \subset O(a)$, $P_A$ doesn't contain a $5$-arc so $P_A \neq S$.

By Proposition \ref{P34}, we can rule out the existence of a Fano plane in the cycle types listed.
\end{proof}

As in the proof of this Lemma, if we can argue that a set of points either has to lie on a line or form a $4$-arc, then Proposition \ref{P34} allows us to rule out some cycle types immediately.

\begin{lemma}\label{L36}
    Take some $a \in D_4$, $O(a) \subset P$ for some Fano plane $P$ if and only if $O(a)$ forms a $4$-arc furthermore $P \in C_7^\lambda$ where $\lambda = (1234)(56)$.
\end{lemma}
\begin{proof}
    If $\{a, F^ia, F^ja\}$ are collinear then either $\{a, Fa, F^2a\}$ are on a line or $\{a, Fa, F^3a\}$. Call these lines $\ell_1$ and $\ell_2$ respectively. Notice $|\ell_1 \cap F\ell_1| = 2$ and $|\ell_2 \cap F\ell_2| = 2$. Thus either $O(a)$ on a line or forms a $4$-arc.

    Clearly if $O(a) \subset P$ then $O(a)$ cannot have four points on one line, thus $O(a) = A$ forms a $4$-arc. Since $A$ is a $4$-arc then $P_A = P$. We know from the Fano axiom that for any $4$-arc $A = O(a)$ $A \subset P$ for a unique Fano plane $P$, which is $P_A$.
    
    Now we can analyze $P_A$, the three intersection points in $P_A$ are $q_1 = \langle \langle a, Fa \rangle, \langle F^2a, F^3a \rangle \rangle$, $q_2 = \langle \langle Fa, F^2a \rangle, \langle a, F^3a \rangle \rangle$, and $q_3 = \langle \langle a, F^2a \rangle, \langle Fa, F^3a \rangle \rangle$. Notice that $Fq_1 = q_2$, $Fq_2 = q_1$, and $Fq_3 = q_3$. Thus $P_A$ is $F$-invariant and the action of $F$ on $P_A$ gives the cycle type $\lambda = (1234)(56)$, as desired.
\end{proof}

\begin{lemma}\label{L37}
    Take some $a, b \in D_2$, $a \notin O(b)$,  $O(a) \cup O(b) \subset P$ for some Fano plane $P$ if and only if $O(a) \cup O(b)$ forms a $4$-arc furthermore $P \in C_7^\lambda$ where $\lambda = (12)(34)$.
\end{lemma}
\begin{proof}
    Suppose $O(a) \cup O(b)$ did not form a $4$-arc, without loss of generality, suppose $a \in \ell = \langle b, Fb \rangle$. Notice $F\ell = \ell$ so $\ell$ is a $q$-line, and thus $a \in \ell$ implies $Fa \in F\ell = \ell$. Thus $O(a) \cup O(b)$ lie on a line, and thus either $O(a) \cup O(b)$ lies on a line or forms a $4$-arc.

    Since $O(a) \cup O(b) \subset P$, $O(a) \cup O(b) = A$ forms a $4$-arc. We know from the Fano axiom that for any $4$-arc $A = O(a) \cup O(b)$, $A \subset P$ for a unique Fano plane $P$, which is $P_A$. 
    
    Now we can analyze $P_A$. Label $q_1 = \langle \langle a, Fa \rangle, \langle b, Fb \rangle \rangle$, since it is the intersection of two $q$ lines it is a $q$-point. Let $q_2 = \langle \langle b, Fa \rangle, \langle a, Fb \rangle \rangle$ and $q_3 = \langle \langle b, a \rangle, \langle Fa, Fb \rangle \rangle$, notice $Fp_2 = p_2$ and $p_3 = Fp_3$. Thus $P_A$ is $F$-invariant and the action of $F$ on $P_A$ gives the cycle type $\lambda = (12)(34)$ as desired.
\end{proof}

If $\lambda \neq (1234)(56)$ but $S \in C_7^\lambda$ contains $O(a)$ for $a \in D_4$, then if $O(a)$ is on a line, for any $4$-arc subset $A \subset S$, we know that $P_A \neq S$ since $P_A$ doesn't have four points on a line. Now if $A = O(a)$ forms a $4$-arc, then we know that $P_A \neq P$ since the action of $F$ on $P_A$ does not give $\lambda$ from Lemma \ref{L36}. Therefore Proposition \ref{P34} implies $\lambda$ does not admit a Fano plane construction. From Lemma \ref{L37}, identical argument can be made for $\lambda \neq (12)(34)$ where $S \in C_7^\lambda$ contains $O(a) \cup O(b)$ for $a, b \in D_2$, $a \neq b$.

\begin{corollary}\label{C38}
The cycle types $(1234)(567), (1234), (123)(45)(67), (12)(34)(56)$ do not admit Fano planes.
\end{corollary}

We are only left to analyze the following cycle types of $S_7$: $(123)(45)$, $(123)$, and $(12)$.

\begin{lemma}\label{L39}
    Take some $a \in D_3$ and $p_1, p_2 \in D_1$, $p_1 \neq p_2$, either $O(a)$ lies on a $q$-line or the set $O(a) \cup \{p_1, p_2\}$ forms a $5$-arc.
\end{lemma}
\begin{proof}
    Suppose $O(a)$ does not lie on a $q$-line and $B = \{a, Fa, F^2a, p_1, p_2\}$ does not form a $5$-arc. Since $O(a)$ is not on a line either some line contains one point from $O(a)$ and both $\{p_1, p_2\}$ or two points from $O(a)$ and one from $\{p_1, p_2\}$. 
    
    In the first case notice that the line $\langle p_1, p_2 \rangle$ defines a $q$-line so if $O(a) \cap \langle p_1, p_2 \rangle \neq \emptyset$ then $O(a) \subset \langle p_1, p_2 \rangle$ a contradiction. In the second case, without loss of generality suppose that $\{a, Fa, p_1\}$ were on a line $\ell$, then $ \{Fa, p_1\} \subset F\ell$ so $F\ell = \ell$ and $O(a) \subset \ell$ a contradiction. Therefore either $O(a)$ lies on a line or $B$ forms a $5$-arc as desired.
\end{proof}

\subsection{(123)(45)}

Set $\lambda = (123)(45)$, and take any $O(a) \cup O(b) \cup \{p_1, p_2\} = S \in C_7^\lambda$, where $a \in D_3$, $b \in D_2$, and $p_1, p_2 \in D_1$, $p_1 \neq p_2$.

Suppose that $O(a)$ lies on a $q$-line $\ell$. If for some other point $q \in S \setminus O(a)$, $q \in \ell$, then we know that if there exists some $4$-arc $A \subset S$ then $P_A \neq S$ trivially. Consider the case where $\ell \cap S = O(a)$. If $A = \{a, Fa, p_1, p_2\}$ does not form a $4$-arc then $A \subset \ell$. Consider the point $c = \langle \langle a, p_1 \rangle, \langle Fa, p_2 \rangle \rangle$, notice $F^3c= c$ so either $c$ a $q^3$-point or a $q$-point. Therefore $P_A \neq S$.

If $O(a)$ does not lie on a $q$-line then from Lemma \ref{L39} $\{a, Fa, F^2a, p_1, p_2\}$ forms a $5$-arc and for some $4$-arc subset $A \subset S$, $P_A \neq S$ since $P_A$ does not contain a $5$-arc. Therefore, from Proposition \ref{P34}, $\lambda = (123)(45)$ does not admit a Fano plane as desired.

\subsection{(123)}

Set $\lambda = (123)$ and take any $O(a) \cup \{p_1, p_2, p_3, p_4\} = S \in C_7^\lambda$ where $a \in D_3$ and $p_i \in D_1$ where $p_i \neq p_j$ if $i \neq j$.

From Lemma \ref{L39} if $O(a)$ was not contained in a line $\ell$, then $S$ contains a $5$-arc and as argued before $P_A \neq S$ for any $4$-arc subset of $S$. If $O(a)$ was on a line $\ell$, then $c = \langle p_1, p_2 \rangle \cap \ell \in P$. Notice that $c$ is the intersection of two $q$ lines and must be a $q$ point. Thus forcing four points of $P$ on a line. From Proposition \ref{P34} we get that $\lambda = (123)$ does not admit a Fano plane.

\subsection{(12)}

Set $\lambda = (12)$ and take any $O(a) \cup \{p_1, p_2, p_3, p_4, p_5 \} = S \in C_7^\lambda$ where $p_i \neq p_j$ if $i \neq j$.

Suppose  $O(a) \cup \{p_1, p_2\}$  does not define a $4$-arc. If $\{a, p_1, p_2\}$ or $\{Fa, p_1, p_2\}$ are collinear then $O(a) \cup \{p_1, p_2\}$ lie on a line and thus we have $4$-points on a line. Without loss of generality, suppose $O(a) \cup \{p_1\}$ are on a line $\ell$. Then consider the set $O(a) \cup \{p_2, p_3\}$, either this defines a $4$-arc, or three points in the set are collinear. As before, if $\{a, p_2, p_3\}$ or $\{Fa, p_2, p_3\}$ are collinear then $O(a) \cup \{p_2, p_3\}$ lie on a line so $4$-points lie on a line. This forces $p_2$ or $p_3$ on the line $\ell = \langle a, Fa \rangle$, but $p_1 \in \ell$ so there are $4$-points on a line.

Therefore for some $i, j$, if $S$ does not contain $4$-points on a line then $A = O(a) \cup \{p_i, p_j\}$ defines a $4$-arc. Notice that $b = \langle \langle a, p_i \rangle , \langle Fa, p_j \rangle \rangle$ defines a $q^2$-point since $Fb \neq b$ and thus $P_A \neq S$. From Proposition \ref{P34} we get $\lambda = (12)$ does not admit a Fano plane.\\

This proves Theorem \ref{3.2T}.\\

For the following cycle types $(1234)(56)$, $(123)(456)$, and $(12)(34)$ we can show that each Fano plane is uniquely generated by a specific $4$-arc. This allows us to more easily count the number of Fano planes of each cycle type.

\begin{proposition}\label{P310}
    There exists a bijection between Fano plane of cycle type $\lambda = (1234)(56)$ and $4$-arcs of cycle type $(1234)$
\end{proposition}

\begin{proposition}\label{P311}
    There exists a bijection between Fano plane of cycle type $\lambda = (12)(34)$ and $4$-arcs of cycle type $(12)(34)$
\end{proposition}

Proposition \ref{P311} and Proposition \ref{P310} are direct corollaries from Lemma \ref{L36} and Lemma \ref{L37} respectively.

\begin{proposition}\label{P312}
    There exists a bijection between Fano plane of cycle type $\lambda = (123)(456)$ and $4$-arcs of cycle type $(123)$
\end{proposition}
\begin{proof}
    From the Fano axiom we know that every $4$-arc $A$ forms a Fano plane. Consider a $4$-arc of cycle type $(123)$, thus $A = O(a) \cup \{p\}$ for $a \in D_3$ and $p \in D_1$. Now consider the intersection points of $P_A$, $q_1 = \langle \langle a, p \rangle, \langle Fa, F^2a \rangle \rangle$, $q_2 = \langle \langle Fa, p \rangle, \langle F^2a, a \rangle \rangle$, and $q_3 = \langle \langle F^2a, p \rangle, \langle Fa, a \rangle \rangle$. Notice that $Fq_1 = q_2$, $Fq_2 = q_3$, and $Fq_3 = q_1$, therefore $\{q_1, q_2, q_3\} \in O(b)$ for some $b \in D_3$. Thus $P_A = O(a) \cup O(b) \cup \{p\}$ is $F$-invariant, and is of cycle type $\lambda = (123)(456)$. Furthermore, $q_1, q_2, q_3$ are collinear so $O(b)$ lies on a $q$-line, thus $P_A$ only has one $4$-arc of type $(123)$.

    Take any Fano plane $O(a) \cup O(b) \cup \{p\} = P \in C_7^\lambda$, first recall that every $4$-arc $A = O(a) \cup \{p\}$ of type $(123)$ is formed by picking $a$ such that $O(a)$ forms a $3$-arc, and then picking any $q$-point $p$. Thus if either $O(a)$ or $O(b)$ form a $3$-arc then $O(a) \cup \{p\}$ forms a $4$-arc of type $(123)$. If both $O(a)$ and $O(b)$ lie on a line, we know their intersection $p$ must be included in $P$ and is neither in $O(a)$ or $O(b)$ thus four $O(a) \cup \{p\}$ lies ona  line, and thus cannot be a subset of a Fano plane.
\end{proof}

Unfortunately for $\lambda = (1234567)$ we cannot provide a way to count the number of Fano planes of this type, furthermore we do not a priori know if they exist. However, we are able to provide a criteria for their existence.

\begin{proposition}\label{P313}
    Let $\lambda = (1234567)$, $O(a) = S \in C_7^\lambda$ is a Fano plane if and only if either $\{a, Fa, F^5a\}$ or $\{a, Fa, F^3a\}$ are collinear.
\end{proposition}
\begin{proof}
    Take any $a \in D^7$, we know from Bergvall \cite[Lemma 6.6]{bergvall2020equivariant} that either $O(a)$ forms a $7$-arc, $O(a)$ is on a $q$-line, or either $\{a, Fa, F^5a\}$ or $\{a, Fa, F^3a\}$ are collinear. Clearly if $S$ is a Fano plane then $O(a)$ cannot be a $7$-arc and $O(a)$ cannot be on a $q$-line, thus either $\{a, Fa, F^5a\}$ or $\{a, Fa, F^3a\}$ are collinear.
    
    If $\{a, Fa, F^3a\}$ are on a line $\ell$, then we can label $O(a)$ as follows: $p_1 = a$, $p_2 = F^4a$, $p_3 = Fa$, $p_4 = F^6a$, $q_1 = F^5a$, $q_2 = F^2a$, and $q_3 = F^3a$ as in Figure \ref{fig:FP}. Notice that $O(\ell)$ generates the Fano plane. An identical argument shows that $\{a, Fa, F^5a\}$ collinear implies $O(a)$ is a Fano plane.
\end{proof}

Consider the case where $O(a)$ is a Fano plane and $\{a, Fa, F^5a\}$ are collinear. Let $\ell$ be the line containing $\{a, Fa, F^5a\}$. We can see that $\{\ell, F\ell, F^3\ell\}$ share a point. Since we have a bijection between $q^n$-lines and $q^n$-points, using this bijection we can map between Fano planes of the form $\{a, Fa, F^5a\}$ and $\{a, Fa, F^3a\}$.

\section{Counting 7-Arcs}\label{S4}

Recall we defined $C_7^\lambda$ as the collection of unordered arcs, we set $U \subset C_7^\lambda$ as a collection of $7$-tuples which contain all $7$-arcs. We then compute the size of $\Delta \subset U$ for each $\lambda$, where each $S \in \Delta$ has three points collinear. From our definitions, $U \setminus \Delta$ is precisely the ordered $7$-arcs of type $\lambda$, and therefore $U \setminus \Delta = B_7^\lambda$.

We know that the action of $\text{PGL}(3, \Fq)$ is free and transitive on $4$-arcs and commutes with $F_q$, therefore $|B_7^\lambda|$ must be divisible by $|\text{PGL}(3, \Fq)| = (q^2 + q + 1)(q^2 + q)(q^2)(q^2 - 2q + 1)$. 

\subsection{(1234)(56)} By Proposition \ref{P310} Fano planes of type $\lambda$ are generated by a unique $4$-arc of type $(1234)$. Knowing this, we set $U \subset C_7^\lambda$ defined by $S \in U$ if and only if for the point $a \in S$ such that $a \in D_q^4$ we have $O(a)$ forms a $4$-arc. We know there are precisely $(1/4)|\text{PGL}(3, \Fq)| = (1/4)(q^2 + q + 1)(q^2 + q)(q^2)(q^2 - 2q + 1)$ $4$-arcs of type $(1234)$.

For each $S \in U$, we will refer to $a$ as the $q^4$ point, $b$ as the $q^2$ point, and $c$ as the $q$-point, thus $S = O(a) \cup O(b) \cup \{c\}$.

We will characterize $\Delta \subset U$ by four subsets $\Delta_1$, $\Delta_2$, $\Delta_3$, and $\Delta_4$. Formally we say that $S \in \Delta_1$ if and only if $O(b) \cap \langle a, Fa \rangle \neq \emptyset$, $S \in \Delta_2$ if and only if $O(b) \cap \langle a, F^2a \rangle \neq \emptyset$, $S \in \Delta_3$ if and only if $c \in \langle a, F^2a \rangle$, and $S \in \Delta_4$ if and only if $c \in \langle b, Fb \rangle$. 

Intuitively $\Delta_1 \subset \Delta$ contains tuples where the $q^2$ point selected lies on a $q^4$-line drawn $O(a)$, $\Delta_2 \subset \Delta$ is all tuples where this point lies on the $q^2$ line drawn by $O(a)$. We have $\Delta_3 \subset \Delta$ as all tuples where the $q$-point lies on the $q^2$-line drawn by $O(a)$, and $\Delta_4 \subset \Delta$ is points where this point lies on the $q$-line drawn by $O(b)$.

\begin{lemma}\label{L401}
    $\Delta = \Delta_1 \cup \Delta_2 \cup \Delta_3 \cup \Delta_4$
\end{lemma}
\begin{proof}
    It is clear that $\Delta_i \subset \Delta$ for all $i$ thus $\Delta \supset \Delta_1 \cup \Delta_2 \cup \Delta_3 \cup \Delta_4$.

    Now take $S \in \Delta$, since $S \in U$, then $O(a)$ forms a $4$-arc. If three points of $S$ are collinear, then either $O(b) \cup \{c\}$ are on a line, $O(a)$ is on a line formed by $O(b) \cup \{c\}$, or $O(b) \cup \{c\}$ is on a line formed by two points in $O(a)$.

    In the first case, $O(b) \cup \{c\}$ are on a line if and only if $c \in \langle b, Fb \rangle$ so $S \in \Delta_4$. In the second case, if $F^ja \in \ell = \langle b, c \rangle$, since $F^2\ell = \langle F^2b, c \rangle = \ell$ then $F^{j + 2}a \in \ell$ so $c \in \langle F^ja, F^{j + 2}a \rangle$ thus $c \langle a, F^2a \rangle \neq \emptyset$ and $S \in \Delta_3$. Note that $F^ja \in \ell = \langle b, Fb \rangle$ is impossible since $O(a)$ is a $4$-arc and does not intersect on a $q$-line. 

    Consider the third case. As before $\langle F^ia, F^{i+1}a \rangle$ cannot contain a $q$-point. If $b \in \langle F^ia, F^{i+1}a \rangle$ then $O(b) \cap \langle a, Fa \rangle \neq \emptyset$ so $S \in \Delta_1$. If $b \in \langle F^{j}a, F^{j+2}a\rangle$ then $O(b) \cap \langle a, F^2a \rangle \neq \emptyset$, and $S \in \Delta_2$. Lastly if $c \in \ell = \langle F^ja, F^{j+2}a \rangle$, but since $c \in F^j\ell$ for all $j$ then $c \in \langle a, F^2 a\rangle$, which is $\Delta_3$. This proves $\Delta_1 \cup \Delta_2 \cup \Delta_3 \cup \Delta_4 = \Delta$.
\end{proof}

Since $\Delta_i \subset U$, we always begin by arbitrarily picking a $4$-arc. Since this contributes the same factor to the count, we will simply count $8|\Delta_i| / |\text{PGL}(3, \Fq)|$ for ease of counting as explained in Section 2 and shown in Equation \ref{Conf1}. We then count the number of ways to pick $b \in D^2$ and $c \in D^1$ such that $S = O(a) \cup O(b) \cup \{c\} \in \Delta_i$ for each $i$.

$S \in \Delta_1$ if and only if $O(b) \cap \langle a, Fa \rangle \neq 0$. Notice that there can only exist at most $1$ $q^2$-point on the $q^4$-line $\langle a, Fa \rangle$ since $O(a)$ is a $4$-arc. Note also the point $\langle a, Fa \rangle \cap \langle F^2a, F^3a \rangle$ is a $q^2$-point, this must be $b$ or $Fb$. We can then pick any $q$-point in $D^1$. This gives:
\begin{align*}
    8|\Delta_1| / |\text{PGL}(3, \Fq)| = 2(q^2 + q + 1)
\end{align*}

$\langle a, F^2a \rangle$ is clearly a $q^2$ line, and contains $q^2$ many $q^2$-points. We can label this point $b$ or $Fb$, note that both cannot lie on this line, and then pick any $q$-point in $D^1$, so $O(b) \cap \langle a, F^2a \rangle \neq \emptyset$ and $S \in \Delta_2$.
\begin{align*}
    8|\Delta_2| / |\text{PGL}(3, \Fq)| = 2q^2(q^2 + q + 1)
\end{align*}

Notice that there exists at most $1$, $q$-point on any $q^2$-line and the point $\langle a, F^2a \rangle \cap \langle Fa, F^3a \rangle$ is a $q$-point. Picking that point and then pick any $q^2$ point gives $S \in \Delta_3$.
\begin{align*}
    8|\Delta_3| / |\text{PGL}(3, \Fq)| = q^4 - q
\end{align*}

We can pick any $q^2$-point from $D^2$ and then any $q$-point on the $q$-line $\langle b, Fb \rangle$, this gives $c \in \langle b, Fb \rangle$ so $S \in \Delta_4$.
\begin{align*}
    8|\Delta_4| / |\text{PGL}(3, \Fq)| = (q^4 - q)(q + 1)
\end{align*}

Clearly $a = \langle a, Fa \rangle \cap \langle a, F^2a \rangle$, so $O(b) \cap \langle a, Fa \rangle \cap \langle a, F^2a \rangle = \emptyset$ for all choices of $b$ and hence $|\Delta_1 \cap \Delta_2| = 0$

For $S \in \Delta_1 \cap \Delta_3$, the $q^2$-point must be $\langle a, Fa \rangle \cap \langle F^2a, F^3a \rangle$ labeled either $b$ or $Fb$, and the $q$-point must be $\langle a, F^2a \rangle \cap \langle Fa, F^3a \rangle$, thus:
\begin{align*}
    8|\Delta_1 \cap \Delta_3|/ |\text{PGL}(3, \Fq)| = 2
\end{align*}

For $S \in \Delta_1 \cap \Delta_4$, the $q^2$-point must be $\langle a, Fa \rangle \cap \langle F^2a, F^3a \rangle$ labeled either $b$ or $Fb$ and we pick one of the $q + 1$, $q$ points on $\langle b, Fb \rangle$.
\begin{align*}
    8|\Delta_1 \cap \Delta_4|/ |\text{PGL}(3, \Fq)| = 2(q + 1)
\end{align*}

For $S \in \Delta_2 \cap \Delta_3$, the $q$-point must be $\langle a, F^2a \rangle \cap \langle Fa, F^3a \rangle$ and we pick one of the $q^2$ many $q^2$ points on $\langle a, F^2a \rangle$, then labeling it either $b$ or $Fb$.
\begin{align*}
    8|\Delta_2 \cap \Delta_3|/ |\text{PGL}(3, \Fq)| = 2q^2
\end{align*}

For $S \in \Delta_2 \cap \Delta_4$ we pick a $q^2$-point on $\langle a, F^2a \rangle$, then labeling it either $b$ or $Fb$, then any $q$-point on $\langle b, Fb \rangle$.
\begin{align*}
    8|\Delta_2 \cap \Delta_4|/ |\text{PGL}(3, \Fq)| = 2q^2(q + 1)
\end{align*}

For $S \in \Delta_3 \cap \Delta_4$, we pick the $q$-point $\langle a, F^2a \rangle \cap \langle Fa, F^3a \rangle$ and then a $q^2$-point on any of the $q$-lines through this point. There are $q^2 - q$, $q^2$-points on each of the $q + 1$ different $q$-lines, note that both $b$ and $Fb$ are on this line.
\begin{align*}
    8|\Delta_3 \cap \Delta_4|/ |\text{PGL}(3, \Fq)| = (q + 1)(q^2 - q)
\end{align*}

Since $\Delta_1 \cap \Delta_2 = \emptyset$, the only non-empty intersections are $\Delta_1 \cap \Delta_3 \cap \Delta_4$ and $\Delta_1 \cap \Delta_3 \cap \Delta_4$. Each $S \in \Delta_1 \cap \Delta_3 \cap \Delta_4$ defines a Fano plane and $\Delta_1 \cap \Delta_3 \cap \Delta_4$ is empty thus giving:
\begin{align*}
    8|\Delta_1 \cap \Delta_3 \cap \Delta_4|/ |\text{PGL}(3, \Fq)| = 2\\
    8|\Delta_2 \cap \Delta_3 \cap \Delta_4|/ |\text{PGL}(3, \Fq)| = 0
\end{align*}

Through inclusion-exclusion, we are able to compute $|\Delta|/|\text{PGL}(3, \Fq)|$. We know that there are $|U| = (1/8)(q^2 + q + 1)(q^4 - q)|\text{PGL}(3, \Fq)|$ so we compute:
\begin{align}
    |B_7^\lambda| / |\text{PGL}(3, \Fq)| = \frac{1}{8} (q^6 - 4q^4)
\end{align}

Notice that if we set the number of Fano planes to $0$ we recover the count given for this cycle type given by Bergvall \cite[Table 1]{bergvall2020cohomology}. The computation for $7$-arcs in all other cycle types is very similar, thus the counts for the specific $\Delta_i$ will be omitted.

\subsection{(123)(456)}

Recall by Proposition \ref{P312} Fano planes of type $\lambda$ are generated by a unique $4$-arc of type $(123)$. We set $U \subset C_7^\lambda$ defined by $S \in U$ if and only if for the point $a \in S$ such that $a \in D_q^3$ we have $O(a) \cup \{c\}$ forms a $4$-arc.

For each $S \in U$, we will refer to $a,b$ as the $q^3$ points and $c$ as the $q$-point, thus $S = O(a) \cup O(b) \cup \{c\}$.

We characterize $\Delta \subset U$ by four subsets. We say $S \in \Delta_1$ if and only if $O(b) \cap \langle a, Fa \rangle \neq \emptyset$, $S \in \Delta_2$ if and only if $O(b) \cap \langle a, c \rangle \neq \emptyset$, $S \in \Delta_3$ if and only if $O(b)$ lies on a line, and $S \in \Delta_4$ if and only if $O(a) \cap \langle b, Fb \rangle \neq \emptyset$.

\begin{lemma}
    $\Delta = \Delta_1 \cup \Delta_2 \cup \Delta_3 \cup \Delta_4$
\end{lemma}
\begin{proof}
    Clearly $\Delta_1 \cup \Delta_2 \cup \Delta_3 \cup \Delta_4 \subset \Delta$.

    Take any $S \in U$, if we have that three points of $S$ are collinear, then either $O(b)$ is on a line, $O(b)$ is on a line formed two points in $O(a) \cup \{c\}$, or $O(a) \cup \{c\}$ is on a line formed by two point in $O(b)$.

    If $S \in U$ is in the first case, then $S \in \Delta_3$. If the second case is true, either $b \in \langle F^ia, F^ja \rangle$ and thus $O(b) \cap \langle a, Fa \rangle \neq \emptyset$ so $S \in \Delta_1$, or $b \in \langle F^ia, c \rangle$ and $S \in \Delta_2$. If the last case is true, either $a \in \langle F^ib, F^jb \rangle$ so $O(a) \cap \langle b, Fb \rangle \neq \emptyset$ and $S \in \Delta_4$, or $c \in \langle F^ib, F^jb \rangle$ this forces $\{b, Fb, F^2b\}$ collinear so $S \in \Delta_3$. This proves $\Delta = \Delta_1 \cup \Delta_2 \cup \Delta_3 \cup \Delta_4$.
\end{proof}

Following a similar line of computation as before we compute $|\Delta|/|\text{PGL}(3, \Fq)| = (1/18)(q^5 + q^4 + 9q^3 - 10q^2 - 7q - 21)$ and since $|U|/|\text{PGL}(3, \Fq)| = (1/18)(q^6 + q^3 - q^2 - q)$ we get:
\begin{align}
    |B_7^\lambda|/|\text{PGL}(3, \Fq)| = \frac{1}{18}(q^6 - q^5 - q^4 - 8q^3 + 9q^2 + 6q + 12)
\end{align}

\subsection{(12)(34)}
As before we define $U \subset C_7^\lambda$ by $S \in U$ if $O(a) \cup O(b)$ forms a $4$-arc, and we refer to each $S \in U$ by $S = O(a) \cup O(b) \cup \{p_1, p_2, p_3\}$ for $a, b \in D^2$ distinct orbits, and $p_i \in D^1$ for all $i$ all distinct. This is because Proposition \ref{P311} allows us to state each $S \in U$ generates precisely one Fano plane.

We characterize $\Delta$ by five subsets. We say $S \in \Delta_1$ if and only if $p_i \in \langle a, b \rangle \cap \langle Fa, Fb \rangle$ for some $i$, $S \in \Delta_2$ if and only if $p_i \in \langle a, Fb \rangle \cap \langle Fa, b \rangle$ for some $i$, $S \in \Delta_3$ if and only if $p_i \in \langle a, Fa \rangle$ for some $i$, $S \in \Delta_4$ if and only if $p_i \in \langle b, Fb \rangle$ for some $i$, and $S \in \Delta_5$ if and only if $p_1, p_2, p_3$ are collinear. We have a similar Lemma as to before:

\begin{lemma}
    $\Delta = \Delta_1 \cup \Delta_2 \cup \Delta_3 \cup \Delta_4 \cup \Delta_5$
\end{lemma}
\begin{proof}
    Clearly $\Delta_1 \cup \Delta_2 \cup \Delta_3 \cup \Delta_4 \cup \Delta_5 \subset \Delta$.

    Take any $S \in U$, if we have three points collinear, we have that either $p_1, p_2, p_3$ are collinear, a point from $O(a) \cup O(b)$ lies on a line formed by two of $p_1, p_2, p_3$, or $p_i$ lies on a line formed by $O(a) \cup O(b)$.

    In the first case if $p_1, p_2, p_3$ are collinear then $S \in \Delta_5$. The second case is redundant since if a point in $O(a) \cup O(b)$ lies on $\ell = \langle p_i, p_j \rangle$, notice $F\ell = \ell$, so either $O(b)$ or $O(a)$ lies on $\langle p_i, p_j \rangle$, so $p_i$ lies on a line formed by $O(a) \cup O(b)$.

    In the third case, either $p_i \in \langle a, b \rangle$ so $p_i \in \langle Fa, Fb \rangle$ and $S \in \Delta_1$, similarly if $p_i \in \langle Fa, b \rangle$ then $S \in \Delta_2$. If $p_i \in \langle a, Fa \rangle$ we have $S \in \Delta_3$ and lastly if $p_i \in \langle b, Fb \rangle$ then $S \in \Delta_4$. This proves $\Delta = \Delta_1 \cup \Delta_2 \cup \Delta_3 \cup \Delta_4 \cup \Delta_5$.
\end{proof}

As before we compute $|\Delta|/|\text{PGL}(3, \Fq)| = (1/48)(7q^5 + 4q^4 - 15q^3 + 5q^2 + 11q -6)$ and since $|U|/|\text{PGL}(3, \Fq)| = (1/48)(q^2 + q + 1)(q^2 + q)(q^2 + q - 1)$ we get:
\begin{align}
    |B_7^\lambda|/|\text{PGL}(3, \Fq)| = \frac{1}{48}(q^6 - 4q^5 - q^4 + 16q^3 - 6q^2 - 12q)
\end{align}

\subsection{(1234567)}

We refer to each $S \in C_7^\lambda$ by $S = O(a)$ for $a \in D^7$. As mentioned in Section 3, we do not have a nice characterization for Fano planes of cycle type $\lambda$. Thus we let $U = C_7^\lambda$, and we split up $\Delta \subset U$ in two subsets where $S \in \Delta_1$ if and only if $O(a)$ lies on a $q$-line and $S \in \Delta_2$ if and only if $O(a)$ forms a Fano plane.

\begin{lemma}
    $\Delta = \Delta_1 \cup \Delta_2$
\end{lemma}
\begin{proof}
    $\Delta \supset \Delta_1 \cup \Delta_2$ is trivial. Take any $S \in U$, and suppose three points of $O(a)$ are collinear, we have five cases. If $\{a, Fa, F^2a\} \subset \ell$ then $\ell = F\ell$ so $O(a) \subset \ell$ a $q$-line. If $\{a, Fa, F^4a\} \subset \ell$ then $F^3\ell = \ell$ so $\ell$ is a $q$-line and $O(a) \subset \ell$. If $\{a, F^2a, F^4a\} \subset \ell$ then $F^5\ell = \ell$ so $O(a) \subset \ell$ a $q$-line. In all cases mentioned $\ell \in \Delta_1$.

    In the last two cases either $\{a, Fa, F^3a\}$ or $\{a, Fa, F^5a\}$ are collinear and from Proposition \ref{P313} we know that this implies that $S$ is a Fano plane so $S \in \Delta_2$. This proves $\Delta = \Delta_1 \cup \Delta_2$ as desired.
\end{proof}

We know that $\Delta_1$ can be calculated trivially, we pick a $q$-line then a $q^7$-point on this $q$-line. There are $q^2 + q + 1$ total $q$-lines and $q^7 - q$ many $q^7$ points on each such line, thus $|\Delta_1| = (1/7)(q^2 + q + 1)(q^7 - q)$.

Notice that all the lines that make up the Fano plane of type $\lambda$ are all $q^7$-lines. We cannot apply the method of counting previously to this case, since we cannot characterize a point that forms a Fano plane using the tools we have been using so far.

We know that the number of Fano planes in any given field must be represented by $q^d + O(q^{d - \epsilon})$ once dividing by $|PGL(3, \Fq)|$ where $d \ge 0$. Notice that in the trivial case, we know that the number of Fano planes is precisely $30$ after also dividing by $|PGL(3, \Fq)|$. If we suppose $d > 0$, then we could pick a field large enough such that $q^d + O(q^{d - \epsilon}) > 30$ and there are more Fano planes of type $\lambda$ over $\Fq$ than Fano planes of the trivial cycle type over $\mathbb{F}_{q^7}$. This is clearly a contradiction. Thus $d = 0$, and the number of Fano planes of type $\lambda$ has to be constant once we divide by $|PGL(3, \Fq)|$.

This tells us that we can compute this constant using $\mathbb{F}_2$ as the base field and this is sufficient for computing the number of $7$-arcs of type $\lambda$ in all fields $\Fq$ of characteristic $2$. The code used for this computation is deferred to the appendix and is written in the computational software SageMath \cite{sagemath}. 

The algorithm implemented is as follows: we take each $a \in D^7$ and consider the set $O(a)$, we know that from Proposition \ref{P313} that $O(a)$ forms a Fano plane if only if either $\{a, Fa, F^3\}$ or $\{a, Fa, F^5\}$ are collinear, if both are collinear then $O(a)$ lies on a $q$-line. We count the number of $a$ where the condition for a Fano plane holds for $O(a)$ and divide by $|PGL(3, \mathbb{F}_{2})|$.

We computed the constant to be $2$, thus $|\Delta_2| = (2/7)|\text{PGL}(3,\Fq)|$ and we get:
\begin{align}
    |B_7^\lambda|/|\text{PGL}(3,\Fq)| = \frac{1}{7}(q^6 + q^4 + q^3 + q^2 - 1)
\end{align}

This proves Theorem \ref{1.21}.

\section*{Acknowledgements}

I am grateful to Ronno Das for suggesting the problem, helpful discussions, guidance in completing the project. I am thankful for the Summer Undergraduate Research Fellowships program at Caltech for the support. I would like to acknowledge Doru Nicola for helpful comments throughout the project.

\bibliographystyle{amsalpha}
\bibliography{references}

\newcommand{\etalchar}[1]{$^{#1}$}
\providecommand{\bysame}{\leavevmode\hbox to3em{\hrulefill}\thinspace}
\providecommand{\MR}{\relax\ifhmode\unskip\space\fi MR }
\providecommand{\MRhref}[2]{%
  \href{http://www.ams.org/mathscinet-getitem?mr=#1}{#2}
}
\providecommand{\href}[2]{#2}
\begin{thebibliography}{KKL{\etalchar{+}}17}

\bibitem[Ber20a]{bergvall2020equivariant}
Olof Bergvall, \emph{Equivariant cohomology of the moduli space of genus three
  curves with symplectic level two structure via point counts}, European
  Journal of Mathematics \textbf{6} (2020), no.~2, 262--320.

\bibitem[Ber20b]{bergvall2020cohomology}
\bysame, \emph{On the cohomology of the space of seven points in general linear
  position}, Research in Number Theory \textbf{6} (2020), no.~4, 48.

\bibitem[DO21]{das2021configurations}
Ronno Das and Ben O’Connor, \emph{Configurations of noncollinear points in
  the projective plane}, Algebraic \& Geometric Topology \textbf{21} (2021),
  no.~4, 1941--1972.

\bibitem[Gly88]{glynn1988rings}
David~G Glynn, \emph{Rings of geometries ii}, Journal of Combinatorial Theory,
  Series A \textbf{49} (1988), no.~1, 26--66.

\bibitem[ISS95]{iampolskaia1995formula}
Anna~V Iampolskaia, Alexei~N Skorobogatov, and Evgenii~A Sorokin, \emph{Formula
  for the number of [9, 3] mds codes}, IEEE Transactions on Information Theory
  \textbf{41} (1995), no.~6, 1667--1671.

\bibitem[KKL{\etalchar{+}}17]{kaplan2017counting}
Nathan Kaplan, Susie Kimport, Rachel Lawrence, Luke Peilen, and Max Weinreich,
  \emph{Counting arcs in projective planes via glynn’s algorithm}, Journal of
  Geometry \textbf{108} (2017), 1013--1029.

\bibitem[{The}22]{sagemath}
{The Sage Developers}, \emph{{S}agemath, the {S}age {M}athematics {S}oftware
  {S}ystem ({V}ersion 9.6)}, 2022, {\tt https://www.sagemath.org}.

\end{thebibliography}

\section*{Appendix}

The code used to compute how many Fano planes of cycle type $(1234567)$ exist for $\mathbb{F}_2$ is listed below.

\begin{lstlisting}[style=Python]
F = 2
frob = F
P_7 = create_ProjectivePlane(GF((F, 7),x))
Fano = 0

def create_ProjectivePlane(k):
    P = [(0, 0, 1)]
    for a in k:
        P.append((0, 1, a))
        for b in k:
            P.append((1, a, b))
    return P

def check_Point_Equal(p_1, p_2):
    return p_1[0] == p_2[0] and p_1[1] == p_2[1] and p_1[2] == p_2[2]

def frob_Point(p, i):
    power = frob ** i
    return (p[0] ** power, p[1] ** power, p[2] ** power)

def find_Line(p_1, p_2):
    if check_Point_Equal(p_1, p_2):
        return False
    elif p_1[0] != 0:
        if p_2[0] != 0:
            if (p_2[1] - p_1[1]) == 0:
                return (-1 * p_1[1], 1, 0)
            else:
                ap = (p_2[2] - p_1[2]) / (p_2[1] - p_1[1])
                return (p_1[1] * ap - p_1[2], -1 * ap, 1)
        else:
            return ((p_1[1] * p_2[2]) - p_1[2], -1 * p_2[2], 1)
    elif p_2[0] != 0:
        return ((p_2[1] * p_1[2]) - p_2[2], -1 * p_1[2], 1)
    else:
        return (1, 0, 0)

for vector in P_7:
    if not check_Point_Equal(vector, frob_Point(vector, 1)):
        l_1 = find_Line(vector, frob_Point(vector, 1))
        l_3 = find_Line(vector, frob_Point(vector, 3))
        l_5 = find_Line(vector, frob_Point(vector, 5))

        if not check_Point_Equal(l_1, l_3) or not check_Point_Equal(l_1, l_5):
            if check_Point_Equal(l_1, l_3) or check_Point_Equal(l_1, l_5):
                Fano += 1

print(Fano)
\end{lstlisting}

\end{document}